\documentclass[12pt]{article}
\usepackage[T2A]{fontenc}			
\usepackage[utf8]{inputenc}			
\usepackage[english]{babel}	
\usepackage{amsmath}
\usepackage{amssymb}
\usepackage{amsfonts}

\usepackage[dvipsnames, svgnames, x11names]{xcolor}
\usepackage{tikz}
\usepackage[dvipsnames, svgnames, x11names]{xcolor}

\usepackage{amssymb,amsfonts,amsthm,amsmath}
\usepackage{latexsym,multicol,graphicx}
\usepackage{float}

\newtheorem{lemma}{Lemma}
\newtheorem{theorem}{Theorem}
\newtheorem{definition}{Definition}
\newtheorem{corollary}{Corollary}
\newtheorem{proposition }{Proposition}
\newtheorem{rem}{Remark}

\newcommand{\eps}{\varepsilon}
\newcommand{\vphi}{\varphi}

\begin{document}

\title{Сonvergence rate for  weighted polynomial approximation  on the real line}

\author{
Anna Kononova
\footnotemark[0] 
\thanks{This work is supported by Russian Science Foundation grant № 20-61-46016} }

\maketitle

	\abstract {
In this note we study a quantitative version of Bernstein's approximation problem when the  polynomials are dense  in weighted spaces on the real line completing a result of S.~N.~Mergelyan (1960). 
We estimate in the logarithmic scale the error of the weighted polynomial approximation  of the Cauchy kernel.

	}
	
\section{Introduction}

The polynomial approximation problem in weighted spaces of functions on the real line is a classical subject in analysis since  the beginning of the 20th century. 

Let $W:\mathbb R\to [1,\infty]$ be an upper semicontinuous  function  on the real line. We denote by $\mathcal C_W$ the linear space of continuous functions $f:\mathbb R\to\mathbb C$ with finite semi-norm
\begin{equation*}
\displaystyle\|f\|_{\infty,W}:=\sup_{t\in \mathbb R}\left|\frac{f(t)}{W(t)}\right|<+\infty.
\end{equation*}
We can think of $\mathcal C_W$ as of a normed space (passing, as usual, to the quotient space under the standard equivalence relation $f\underset{W}{\sim} g\;\;\Leftrightarrow\;\;\|f-g\|_{\infty,W}=0$). 
Throughout the paper we always assume that $W$ grows to infinity faster than any polynomial:
\begin{equation}\label{Pol}
\displaystyle\lim_{|t|\to\infty}\frac{t^n}{W(t)}= 0\;\;\;\forall n\in\mathbb N.
\end{equation}
This condition  ensures that the space $\mathcal C_W$ contains all polynomials.

In 1924, S.~N.~Bernstein \cite{Bern} posed the following question: {\it for which functions $W$ satisfying (\ref{Pol}) the polynomials are dense in $\mathcal C_W$? 
}The Bernstein weighted approximation problem  has been  persistently attracting attention of analysts for almost a century.  

T.~Hall   \cite{Hall} proved in 1938  that if the  polynomials are dense  in $\mathcal C_W$, then necessarily
\begin{equation}
\label{Hall}\int_{-\infty}^\infty \frac{\log W(t)}{1+t^2}{\rm d}t=\infty.
\end{equation}
This condition fails to be sufficient for the density of polynomials \cite[Section VI.H.3]{Koosis}.

There are different approaches to Bernstein's problem. We mention here two  classical  papers by N.~I.~Akhiezer \cite{A} and S.~N.~Mergelyan \cite{M},  both  published in 1956. 
Let us recall  Mergelyan's solution to Bernstein's problem.
He  introduced the function
$$\Omega_W(z):=\sup\left\{|P(z)|: P\in\mathcal P, |P(t)|\le \sqrt{1+t^2}W(t),\; t\in \mathbb R\right\},\;\;z\in\mathbb C\setminus\mathbb R,
$$
where $\mathcal P$ is the space of  the polynomials. Mergelyan proved that the density of the polynomials  in $\mathcal C_W$ is equivalent to each of the following conditions:
\begin{itemize}
    \item $\Omega_W(i)=\infty,$
\item$\displaystyle \int_{-\infty}^\infty \frac{\log\Omega_W(t)}{1+t^2}{\rm d}t=\infty.$
\end{itemize}

If the function $W$ is such that the polynomials are dense in the space $\mathcal C_W$, it is natural to ask  about the approximation rate by polynomials. More precisely, let $\mathcal P_n$ denote the space of the polynomials of degree less than or equal to $n$. For a function $f \in \mathcal C_W$ and for positive $n$, we can define the error  of approximation by polynomials of  degree $n$ by

$$ \mathcal E_n(f) = \inf_{P \in \mathcal P_n} \| f - P\|_{\infty, W}.
$$
The asymptotics of the sequence $\{\mathcal E_n(f)\}$  for various functions $f$ were studied by numerous authors, and we refer the reader to the survey papers of D.~Lubinsky \cite{L} and of  H.~N.~Mhaskar \cite{Mha} on this subject.

In this paper, we concentrate on the case when the function $f$ to be approximated is fixed and equals the  Cauchy kernel $\mathcal K(x):= ({x-i})^{-1}$. For a particular class of functions $f$, the values $\mathcal E_n(f)$  were estimated by N.~I.~Akhiezer \cite{A}; G.~Wahde  \cite{Wahde}  found  estimates of such kind in the  $L^2$ norm and  used them to deal with the uniform weighted approximation problem. M.~M.~Dzhrbashyan  \cite{Dzr2}  studied the  best approximation  error of the Cauchy kernel by rational functions.

In 1960, Mergelyan \cite{Merg}
found an upper bound  for  $\mathcal E_n (\mathcal  K)$ for quite a wide class of functions $W$.
The main goal of the present paper is   to obtain  matching upper and lower bounds for $\mathcal E_n(\mathcal K)$ in the logarithmic scale. An analogous problem in $L^p$ norm is also considered.
The proofs use the ideas  from the paper \cite {BKS} by A. Borichev, M. Sodin, and the author.

\subsubsection*{Acknowledgements}
The author is  grateful to A.~Borichev for attracting attention to this problem and for constructive suggestions, to   E.~Abakumov and M.~Sodin for their helpful remarks and recommendations.

\section{Main result}
\begin{definition}
Let  the function $W$ be of the form 
\begin{equation*} \label{W}W(x)=\exp(\vphi(|x|)),\;\;x\in\mathbb R,
\end{equation*}
 where $\vphi$ is a positive continuous   function strictly increasing on $\mathbb R_+$.  
 We will say that a function $W$ is {\it a weight} if it satisfies  conditions (\ref{Pol}) and (\ref{Hall}). 
 \end{definition}

 Given $n\ge \vphi(0)$, define $$A_n:=\vphi^{-1}(n).$$
 
 Given 	$\alpha\in\mathbb R$ we  introduce the perturbed weight 
 $$ W_\alpha(x):={W(x)}{(x^2+1)^{\alpha/2}},$$
 and the corresponding function $ \vphi_\alpha:=\log W_\alpha$.

 We deal with the space  $\mathcal C_W$  equipped with the norm $\|f\|_{\infty,W}$ defined in the introduction, and  with the weighted $L^p$ spaces defined as follows:
  $$L^p_W:=\left\{f:\mathbb R\to\mathbb C: \|f\|_{p,W}:=
 \displaystyle\left(\int_{\mathbb R} \left|\frac{f(x)}{W(x)}\right|^p{\rm d} x\right)^{1/p} <\infty
 \right\},$$
 where $p\in [1,\infty).$
For $1\le p\le\infty$ and $n\in \mathbb N,$ define
$$E_n({p,W}):=\inf_{Q\in\mathcal P_{n-1}}\left\{\left\|\frac 1{x-i}-Q(x)\right\|_{p,W}\right\}.
$$



%


The sequence $\left(E_n(p,W)\right)_n$ is nonincreasing; if the polynomials are dense in $L^p_W$ $(\mathcal C_W),$ then this sequence tends to $0$. 
We are interested in  estimating  the growth rate of the sequence $|\log E_n(p,W)|$ as $n\to \infty$ in terms of the function $\vphi$. 

An easy calculation shows that 
$$E_n({p,W})=\inf_{P_n(i)=1,P_{n}\in\mathcal P_n}\left\|{\frac{P_{n}(x)}{\sqrt{x^2+1}}}\right\|_{p,W}=\inf_{P_n(i)=1,P_{n}\in\mathcal P_n}\left\|{{P_{n}(x)}}\right\|_{p,W_{1}}.$$
In particular, for the case $p=2$ our results relate to the asymptotical properties of the Christoffel function outside the real line (we refer to \cite{LL} for more information about the Christoffel function).
On the other hand, in terms of Mergelyan's function $\Omega_W$ we have
$$\lim_{n\to\infty}E_n({\infty, W})=\frac 1 {\Omega_W(i)}.$$

In our note we deal with the following classes of functions.

\begin{definition}
Given a continuous increasing function  $\vphi: \mathbb R_+\to \mathbb R_+$ satisfying  condition (\ref{Hall}),  we say that $\vphi$ is
\begin{itemize}
    \item{\bf normally growing } if  $\vphi(x)/x^{2}$ is decreasing and $\vphi(x)$ is a convex function of $\log x$ on $[A,+\infty)$  for some $A>0$;
     \item {\bf rapidly growing} if $\vphi(x)/x^{1+\eps}$  is increasing on $[A,+\infty)$ for some $\eps>0$ and $A>0$;
     \item {\bf regularly growing} if it is either normally growing or rapidly growing.

\end{itemize}
\end{definition}

\begin{rem}{\rm
There is a nonempty intersection between the classes of rapidly growing and regularly growing functions, e.g., the function $\vphi(x)=x^{3/2}$ satisfies both conditions.
}\end{rem}

\begin{rem}{\rm
Note that the polynomials are dense in $\mathcal C_W$ and in $L^2_W$ provided that $\log W$ is a regularly growing function \cite[Sections VI.D, VI.G]{Koosis}. For normally growing weights, this  can be proved  using the convexity of $\vphi(e^t)$, and for rapidly growing weights this follows from the fact that $W(x)\gtrsim e^{|x|}$ for $|x|$ large enough.
Therefore,  for a regular weight $W$ the sequence $(E_n(p,W))$ tends to $0$. In particular, for all but finitely many $n\in\mathbb N$ we have $\log E_n<0$.   }
\end{rem}

We use the following notation: given two  positive (or two negative) sequences $(\alpha_n)$ and $(\beta_n)$   we will write
\begin{itemize}
\item $\alpha_n\lesssim \beta_n $ if for some  constant $C>0$ one has $\alpha_n\le C\cdot\beta_n$ for $n\in\mathbb N$ and
\item $\alpha_n\simeq \beta_n$ if both $\alpha_n\lesssim \beta_n$ and $\beta_n\lesssim \alpha_n$ are true.
\end{itemize}

Analogous notation is  used  for functions.


The main result of our paper is
\begin {theorem}\label{main}
	Suppose that $\vphi$ is a regularly growing function. Given  $p\in [1,\infty]$, we have
	$$\displaystyle\log E_n(p,W)\simeq-  \int_{0}^{1} \min\left(\vphi		\left(\frac1x\right),\ n\right){\rm d }x.$$	
	
\end{theorem}

The implicit constants may depend on $\vphi$ but are independent of $n$.

\begin{figure}[H]
\begin{center}
\begin{tikzpicture}[scale=1.2]

\draw[ domain=0.5:2] plot (\x, {pow(\x, -2)});
\fill[ fill=gray!25, domain=0.67:2] (0.67,0)--(0.67,2.25)--plot (\x, {abs(pow(\x, -2))})--(2,0.25)--(2,0)--cycle;
\fill[ fill=gray!25, domain=-2:-0.67] (0,0)--(0,2.25)--(0.67,2.25)--(0.67,0)--cycle;
\draw[very thick , domain=0.67:2]plot (\x, {pow(\x, -2)}); 
\draw[very thick] (0,2.25)--(0.67,2.25);
[>=stealth]\draw[thick,->] (-0.1,0) -- (2.3,0);  \draw[thick, ->] (0,-0.1) -- (0,4.1);
\draw[dashed] (0.67,0) -- (0.67,2.25);
\draw(-.5, 2.25)node[above right] {$n$}--(2,2.25);
\node at (1,3) {$\varphi (\frac{1}{x})$};
\end{tikzpicture}
\caption{$\min\left\{\displaystyle\vphi\left(\frac 1x\right), n\right\}$.}
\end{center}
\end{figure}


	\begin{rem}{\rm
The geometrical meaning of our growth classes can be illustrated  by Figure 1. If the function $\vphi$ grows rapidly, then the area of the part  of the  subgraph of the function $\min\left\{\displaystyle\vphi\left(\frac 1x\right), n\right\}$ under the cut-off grows not slower than the area of the remaining part of the subgraph as $n$ tends to infinity, while in the case of the normal growth the part under the cut-off grows not faster than the area of the remaining part of the subgraph.}

\end{rem}



%
%

The following Corollary can be  verified by a simple calculation, which we skip.

\begin{corollary}
Let $p\in[1,\infty].$
Denote $E_n:=E_n(p,W)$. Then we have the following estimates.
\begin{itemize}
 \item If $\vphi(x)=\displaystyle\frac x{\log(2+x)},$
then 
\begin{align*}
 \left|\log{E_n}\right|\simeq  \log \log (n+e).
\end{align*}
 \item If $\vphi(x)= x\log^\nu (2+x), \nu>-1,$
then 
\begin{align*}
 |\log{E_n}|\simeq  \log^{\nu+1} n.
\end{align*}
 \item If $\vphi(x)= x^\nu,\;\;\nu>1,$ then
\begin{align*}
 |\log{E_n}|\simeq   n^{1-1/\nu}.
\end{align*}
 \item If $\vphi(x)= \exp({x^\nu}),\;\;\nu>0,$
then 
\begin{align*}
 |\log{E_n}|\simeq  \frac n{(\log n)^{1/\nu}}. 
\end{align*}
\end{itemize}

\end{corollary}

In Section 3 we establish some properties of the Tchebyshev polynomials and bring some results by Videnskii and Mergelyan. In Section 4 we prove Theorem \ref{main}, first in the uniform case and then in the $L^p$ case.

\section{Preliminaries}

\subsection{Some properties of the Tchebyshev polynomials}

In this subsection we have collected some properties of Tchebyshev polynomials that are used in our paper.
We start with the classical Tchebyshev inequality.

Let $T_n$  denote the Tchebyshev polynomial of the first kind of degree $n$:
$$T_n(x):=\frac12\left(\bigl(x+\sqrt{x^2-1}\bigr)^n+\bigl(x-\sqrt{x^2-1}\bigr)^n\right).
$$

\begin{flushleft}
{\bf Tchebyshev inequality. }\label{T}{\it
For any polynomial $P_n$ of degree $n$  such that $|P_n|\le 1$ on $[-1, 1],$ we have
$$|P_n|\le |T_n|\;\; \mbox{on} \;\;\mathbb R\setminus (-1,1).$$}
\end{flushleft}
%
We   also need two technical results.

\begin{lemma}\label{Tcheb1}
Let $n=2k,\;\;k\in\mathbb N,$ and $a\in\mathbb R$. Then
$$|T_n(i/a)|\ge \frac 12
\left(\frac 1{a}+1\right)^{n}.$$ 
\end{lemma}

\begin{proof}

\begin{align*}
|T_n(i/a)|&=\frac 12\left| \left(\frac ia+\sqrt{-\frac 1{a^2}-1}\right)^n+\left(\frac ia-\sqrt{-\frac 1{a^2}-1}\right)^n\right|
\\&=\frac 12 \left(\frac 1a+\sqrt{\frac 1{a^2}+1}\right)^n+\frac 12\left(\frac 1a+\sqrt{\frac 1{a^2}+1}\right)^{-n}
\ge \frac 12
\left(\frac 1{a}+1\right)^{n}.
\end{align*}

\end{proof}

\begin{lemma}\label{Tcheb2}
Let  $\vphi$ be a rapidly growing function. Then
  \begin{align*}
 \sup_{x\ge A_n}\left\{ \frac{|T_n(x/A_n)|}{W(x)}\right\}
\le \frac{2^n}{e^n}.
\end{align*}
\end{lemma}
\begin{proof}

We have
\begin{align*}
\left|T_n(x/A_n)\right|=\frac 12\left| \left(\frac x{A_n}+\sqrt{\frac {x^2}{A_n^2}-1}\right)^n+\left(\frac x{A_n}-\sqrt{\frac {x^2}{A_n^2}-1}\right)^n\right|
\\\le  \left(\frac x{A_n}+\sqrt{\frac {x^2}{A_n^2}-1}\right)^n\le  \left(2\frac x{A_n}\right)^n,\qquad x\ge A_n.
 \end{align*}
 Therefore,
 \begin{align}\label{10}
 \sup_{x\ge A_n}\left\{ \frac{|T_n(x/A_n)|}{W(x)}\right\}
\le \frac{2^n}{A_n^n} \sup_{x\ge A_n}\left\{\frac{x^n}{{W(x)}}\right\}.
 \end{align}
Since the function $\vphi$ is rapidly growing, for some $\eps>0$ and for large $n$ we have
 $$\displaystyle \log \frac{x^n}{{W(x)}}=n\log x-\vphi(x)\le n\log x- \frac{\vphi(A_n)}{A_n^{1+\eps}}{x^{1+\eps}}=:g_n(x).$$


Furthermore, 
$$g_n'(x)=\frac nx-({1+\eps})x^{{\eps}}\frac{n}{A_n^{1+\eps}},
$$
and we conclude that the only critical point $x^*$ of the function $g_n$ satisfies the relation 
$$\left(\frac {x^*}{A_n}\right)^{1+\eps}=\frac 1{1+\eps}$$
and, hence, the function $g_n$ decreases on $[A_n,\infty)$.
Therefore,
$$\log\frac{x^n}{W(x)}\le g_n(x)\le g_n(A_n)=n\log A_n-n$$
for $x\ge A_n.$
Combining this with (\ref{10}), we get
 \begin{align*}
 \sup_{x\ge A_n}\left\{ \frac{|T_n(x/A_n)|}{W(x)}\right\}
\le \frac{2^n}{e^n}.
\end{align*}

\end{proof}

	\subsection{Results of Mergelyan and Videnskii}
		The result of Mergelyan already mentioned in the introduction is based on a lemma by Videnskii (\cite{V}, Lemma \ref{LV} below). For the reader's convenience, we provide here the proofs of the  versions of both results, which suffice for our purposes.

\begin{theorem}[Mergelyan, \cite{Merg}]\label{ML}
Let $\vphi$ be a  function  satisfying (\ref{Hall}) and such that $\vphi(t)$ is a convex function of $\log t$. Then we have 
$$\log E_n(\infty, W)\lesssim -\int_{0}^{l_n} \frac{\vphi(x){\rm d} x}{x^2+1},$$
with $\displaystyle l_n=\frac{A_{2n}}{2e}$.
\end{theorem}

\begin{lemma}[Videnskii, \cite{V}]\label{LV}
Let $\vphi (t)$  be a convex function of $\log t$. Set $$M_k=\sup_{x>0}\frac{x^{2k}}{W(x)}$$ and
 $$F(x):=\sum_{k=0}^\infty \frac{x^{2k}}{2^kM_k}.$$
Then $ F(x)\lesssim W(x)\lesssim x^2 F(2x),\;x\ge 1.$

\end{lemma}
\begin{rem} {\rm
Note that this lemma  can also be derived from the results of \cite{AD}.}
\end{rem}

\subsubsection{\it Proof of the Videnskii lemma}

The sequence $\left(M_k\right)_k$  increases for $k$ sufficiently large. Set
$$
T(x):=\sup_{k\ge 0}\frac{x^{2k}}{M_k}, \qquad x>0.
$$
Since the function $t\mapsto\vphi(\exp t)$ is convex, the graph of the function
$$\log T(\exp t)=\sup_{k\ge 0}\left(2kt-\log M_k\right)$$
is an infinite polygon consisting of the supporting lines of $\vphi(\exp t)$ with even slopes. Therefore, we have $T(x)\le W(x)$ for large $|x|$.
Then $$F(x):=\sum _{k\ge 0}\frac{x^{2k}}{2^{k}M_k}\le \sum _{k\ge 0}\frac{T(x)}{2^{k}}\lesssim W(x),\qquad x\in\mathbb R.$$

Next, let $Kt-B$ be (some) supporting line of the graph of the convex function $\vphi(\exp t)$ at a sufficiently large point $t^*$:
$$\vphi(\exp (t))\ge Kt-B,\;\;\vphi(\exp (t^*))=Kt^*-B,\qquad K,B\in\mathbb R^+.$$
If the slope $K$ is even, then $B=\log M_{K/2}$ and $T(\exp (t^*))=W(\exp (t^*)).$ Otherwise, let $m$ be the integer part of $K/2$. Then  $\log M_m\le B$ and we have 
\begin{multline*}
\vphi(\exp (t^*))= Kt^*-B\le Kt^*-\log M_m \le (2m+2)t^*-\log M_m\le \log \left(T(\exp (t^*))\right)+2t^*,    
    \end{multline*}
that is 
$$W(x)\le  x^2T(x)$$
for  sufficiently large $x=\exp(t^*)$. Thus  we obtain
$$F(x)=\sum _{k\ge 0}\frac{x^{2k}}{2^{k}M_k}\ge\sum _{k\ge 0}\frac{x^{2k}}{2^{2k}M_k}\ge T\left(\frac x2\right)\gtrsim {x^{-2}W(\frac x2)},
$$
which  proves Lemma.
$\hfill\qed$

\subsubsection{\it Proof of the Mergelyan theorem.}

We set $\displaystyle M_k=\sup_{x>0}\frac{x^{2k}}{W(x)},$ and $b_k:=\displaystyle\frac 1{2^kM_k},$ so that $$F(x)=\sum _{k\ge 0}b_kx^{2k}.$$
Fix $n\ge \vphi(0)$.
By the definition of $M_k$ it is clear that 
$b_k\displaystyle\le \frac{W(x)}{2^kx^{2k}}$ for every $x>0$. Taking $x=A_{2n}$ we see that
$$b_k\le \frac{e^{2n}}{2^kA_{2n}^{2k}}.$$

 Consider the polynomial $\displaystyle P_{2n}:=\sum_{k=0}^{n} b_k x^{2k}$. For  $\displaystyle |x|\le \frac{A_{2n}}{e},$
 we have
 \begin{align*}
     0\le F(x)-P_{2n}(x) \le \sum_{k=n+1}^\infty b_k x^{2k}\le \sum_{k=n+1}^\infty\frac{e^{2n}}{2^kA_{2n}^{2k}}{x^{2k}}\le \sum_{k=n+1}^\infty\frac{e^{2n}}{2^ke^{2k}}\le 1/2.
 \end{align*}
 Since $M_0\le 1$, we have $b_0\ge 1$, $F(x)\ge 1$, $x\in \mathbb R$, and, hence, 
 \begin{align}\label{HP}
     P_{2n}(x) \ge F(x)/2, \qquad |x|\le \frac{A_{2n}}{e}.
 \end{align}

Let $Q_n$ be a polynomial of degree $n$, with no zeros in the upper half plane and such that $|Q^2_n(x)|=P_{2n}(x),\;\; x\in\mathbb R.$ Then
$$F(x)\ge |Q^2_n(x)|=P_{2n}(x)\ge 1.$$
By (\ref{HP}) we have
$$|Q_n^2(x)|=P_{2n}(x)\ge F(x)/2,\qquad |x|\le \frac{A_{2n}}{e}.$$
By the Poisson formula and the Videnskii lemma we obtain that ($\displaystyle l_n:=\frac{A_{2n}}{2e}$):
 \begin{align*}\log |Q_n^2(i)|=\frac 1 {\pi}\int_{-\infty}^{+\infty}\frac{\log |Q_n^2(x)|}{x^2+1}{\rm d}x\gtrsim\int_{0}^{2l_n}\frac{\log |Q_n^2(x)|}{x^2+1}{\rm d}x\\
 \gtrsim  \int_{0}^{2l_n}\frac{\log F(x)}{x^2+1}{\rm d}x +O(1)\gtrsim\int_{0}^{2l_n}\frac{\log W(\frac x2)-2\log x}{x^2+1}{\rm d}x+O(1)\\\gtrsim \int_{0}^{2l_n}\frac{\vphi(x/2)}{x^2+1}{\rm d}x+O(1)\gtrsim \int_{0}^{l_n}\frac{\vphi(x)}{x^2+1}{\rm d}x,\qquad n\to\infty.
 \end{align*}
 
 By the definition of $E_n(\infty, W)$ we see that
 \begin{align*}\log E_n({\infty, W})\le \log\left\|\frac{Q_n}{Q_n(i)}\right\|_{\infty,{W_1}}\le \log\|Q_n\|_{\infty,{W}}-\log{|Q_n(i)|}\\\lesssim \log\|F\|_{\infty,W}-\int_{0}^{l_n}\frac{\vphi(x)}{x^2+1}{\rm d}x\le -\int_{0}^{l_n}\frac{\vphi(x)}{x^2+1}{\rm d}x.
  \end{align*}
  which  proves the Mergelyan theorem.
  $\hfill\qed$

\section{Proof of the main theorem}
\subsection{An equivalent formulation of the main theorem}

We start with a reformulation of Theorem 1 which clarify a bit our estimates.

	\begin {theorem}\label{main2}
Under the hypothesis of  Theorem \ref{main}, for $1\le p\le\infty,$ we have
	\begin{itemize}
\item [(a)]
$\displaystyle\log E_n(p,W)\simeq 
-\frac{n}{\vphi^{-1}(n)}$
if $\vphi$ grows rapidly;
	\item [(b)]
	$\displaystyle\log E_n(p,W)\simeq -\displaystyle\int_{0}^{\vphi^{-1}(n)} \frac{\vphi(x){\rm d} x}{x^2+1}$
	if $\vphi$ grows normally.
\end{itemize}
	\end{theorem}

To show how Theorem \ref{main2} implies Theorem \ref{main}   we use the following lemma.
 
\begin{lemma}\label{slow&fast}
If a function $\vphi$ is  normally growing, then 
 $$
 \int_{0}^{a} \frac{\vphi(x){\rm d} x}{x^2+1}\gtrsim  \frac{\vphi(a)}{a},\qquad a\ge 1,
 $$
while for a rapidly growing function $\vphi$
we have \begin{align*}
 \int_{0}^{a} \frac{\vphi(x){\rm d} x}{x^2+1}\lesssim  \frac{\vphi(a)}{a},\qquad a\ge 1.
\end{align*}
\end{lemma}
\begin{proof}
If  the function $\vphi(x)/x^{2}$ decreases for $x\ge A$, then we have
\begin{align*}
\int_{0}^{a} \frac{\vphi(x){\rm d} x}{x^2+1}\ge\int_{A}^{a} \frac{\vphi(x){\rm d} x}{x^2+1}\ge \frac{\vphi(a)}{a^{2}}
\int_{A}^{a} \frac{x^{2}{\rm d} x}{x^2+1}\gtrsim 
 \frac{\vphi(a)}{a},\qquad a\ge 2A.
 \end{align*} 
 If $\vphi(x)/x^{1+\eps}$ inreases for $x\ge A$ and for some $\eps>0$, then
\begin{align*}
\int_{0}^{a} \frac{\vphi(x){\rm d} x}{x^2+1}\lesssim \int_{A}^{a} \frac{\vphi(x){\rm d} x}{x^2+1}\lesssim \frac{\vphi(a)}{a^{1+\eps}}
\int_{A}^{a} \frac{x^{1+\eps}{\rm d} x}{x^2+1}\lesssim 
 \frac{\vphi(a)}{a},
 \qquad a\ge A.\end{align*}

\end{proof}

Combining Theorem 3 with Lemma 4 in both cases we get 
$$-\log E_n\simeq\int_{0}^{A_n} \frac{\vphi(x){\rm d} x}{x^2+1}+ \frac{\vphi(A_n)}{A_n}.$$
Changing the variables we obtain
\begin{multline*}-\log E_n\simeq\int_{0}^{A_n} \frac{\vphi(x){\rm d} x}{x^2+1}+ \frac{\vphi(A_n)}{A_n}\simeq\int_{1}^{A_n} \frac{\vphi(x){\rm d} x}{x^2}+ \frac{n}{A_n}=\int_{0}^{1} \max\left(\vphi		\left(\frac1x\right),\ n\right){\rm d }x,
\end{multline*}
which yields  the statement of Theorem \ref{main}.\hfill\qed

\subsection{Estimating  $\log E_n$ for the uniform norm}

\subsubsection{The lower bound}

Here we prove the following result.
\begin{lemma}\label{LBL} { Given a weight $W$  we have }
\begin{align*}
 \log{E_n}\gtrsim -\int_{0}^{A_n} \frac{\vphi(x)\rm d x}{x^2+1}
-\frac{n}{A_n},
\end{align*}
where $E_n=E_n(\infty,W).$
\end{lemma}

\begin{proof}

Let $P_n$ be an extremal polynomial of degree $n$ such that $P_n(i)=1$ and $E_n=\|P_n\|_{\infty,W_1}$.
Then 
\begin{equation}\label{2}
    0\le\int_{\mathbb R} \frac{\log|P_n(x)|{\rm d} x}{x^2+1}= \left( \int_{|x|<a}+\int_{|x|\ge a}\right)\frac{\log|P_n(x)|{\rm d} x}{x^2+1}=:I_1+I_2,
\end{equation}
where $a>1$ is some parameter to be chosen later.

We start by estimating  the first integral:
\begin{multline}\displaystyle\label{1}
I_1=\int_{-a}^a \frac{\log|P_n(x)|}{x^2+1}{\rm d} x\\
 \le 
\int_{-a}^a\log\left|\frac{P_n(x)}{\sqrt{x^2+1}W(x)}\right| \frac{{\rm d} x}{x^2+1}
+\int_{0}^a \frac{\log{(x^2+1)}}{x^2+1}{\rm d} x
+
2\int_{0}^a \frac{\vphi(x)}{x^2+1}{\rm d} x\\
\lesssim\int_{-a}^a \log{\left\|{P_n(x)}\right\|_{\infty,W_1}}\frac{{\rm d} x}{x^2+1}+ 
\int_{0}^a \frac{\vphi(x)}{x^2+1}{\rm d} x
\lesssim \log{E_n}+ 
\int_{0}^a \frac{\vphi(x)}{x^2+1}{\rm d} x, 
\end{multline}
for $a$ large enough.

The next step is to estimate the second integral  in (\ref{2}) with the help of the  Tchebyshev inequality (Lemma \ref{T}):

\begin{align*}\displaystyle\label{2}
I_2
\le &
\int_{a}^\infty \frac{\log{|T_n(x/a)|}}{x^2+1}{\rm d} x+ 
\int_{a}^\infty  \frac {\log \left(\max_{[-a,a]}|P_n|\right)}{x^2+1}{\rm d} x
\\\lesssim & \int_{a}^\infty \frac{n\cdot\log(x/a)}{x^2+1}{\rm d} x
+
\int_{a}^\infty  \frac {\log \left(\max_{[-a,a]}|P_n|\right)}{x^2+1}{\rm d} x\\
\lesssim& 
\frac na + \int_{a}^\infty {\log \left\| {P_n}\right\|_{\infty, W_1} }\frac{{\rm d} x}{x^2+1}
+\int_{a}^\infty  \frac {\log \left(\sup_{[-a,a]} W_{1}\right)}{x^2+1}{\rm d} x\\
\lesssim &\frac na+\log E_n+\frac 1a{\log {\left( W_{1}(a)\right)}}
\lesssim \frac na+\log E_n+\frac{\vphi(a)}{a}.
\end{align*}

By (\ref{2}) and (\ref{1}) we conclude that
\begin{align*}
- \log{E_n} \lesssim\int_{0}^a \frac{\vphi(x)}{x^2+1}{\rm d }x
+\frac{n+\vphi(a)}a.
\end{align*}

Finally, let $a=A_n=\vphi^{-1}(n)$. Then
$$
- \log E_n \lesssim\int_{0}^{A_n} \frac{\vphi(x)}{x^2+1}{\rm d }x
+\frac{n}{A_n},
$$
which proves the lemma. 


\end{proof}

\subsubsection{The upper bound}

For normally growing $\vphi$ we just use  the Mergelyan Theorem.
 Indeed, 
 since the function $\vphi(x)/x^{2}$ decreases for $x$ large enough, we have
\begin{equation}\label{2e}
    \int_{0}^{A_{2n}/(2e)}  \frac{\vphi(x)\rm d x}{x^2+1}\gtrsim\int_{0}^{A_{n}}
\frac{\vphi(x){\rm d} x}{x^2+1}.
\end{equation}
 To obtain an upper bound in the case of  rapidly growing functions $\vphi$ we use the Tchebyshev polynomials. 
Taking into account Lemma \ref{Tcheb1},  Lemma \ref{Tcheb2}, and the fact that 
$$\sup_{0<x\le A_n}\left\{ \frac{|T_n(x/A_n)|}{W(x)}\right\}\le 1,$$
we   get by the definition of $E_n$:
 
\begin{multline*}
E_n\le \frac{\|T_n(x/A_n)\|_{\infty,W_1}}{|T_n(i/A_n)|}\le\frac{\|T_n(x/A_n)\|_{\infty,W}}{|T_n(i/A_n)|}\\ \le
  \frac 1{|T_n(i/A_n)|}\max\left(1,\sup_{x>A_n}\left\{ \frac{|T_n(x/A_n)|}{W(x)}\right\}\right)\\
\lesssim \max
\left(1, \frac{2^n}{e^n}\right)\left(\frac 1{A_n}+1\right)^{-n}\le \left(\frac 1{A_n}+1\right)^{-n},
\end{multline*}
and finally
\begin{equation}\label{upper}\log E_n\lesssim -\frac n{A_n}.
\end{equation}
 
 \subsubsection{Conclusion}
 Estimates  (\ref{2e}), (\ref{upper}) together with Lemma \ref {slow&fast} and Lemma \ref{LBL} give Theorem \ref{main2} in the uniform case.

\subsection {Estimating  $\log E_n$ for the   weighted $L^p$ space, $1\le p<\infty$}

\subsubsection{The upper bound}

Given $1\le p<\infty, n\ge 0,$ we choose a polynomial $P_n$ of degree $n$ such that $$E_n(\infty,W_{-2})=\left\|\frac{P_n}{\sqrt{x^2+1}}\right\|_{\infty, W_{-2}}=\left\|P_n\right\|_{\infty,W_{-1}}$$ and $|P_n(i)|=1.$ Then
\begin{multline*}
E_n^p(p,W)\le \left\|\frac{ P_n}{\sqrt{x^2+1}}\right\|_{p,W}^p=\int_{\mathbb R} \left|\frac{P_n(x)}{W(x)\sqrt{x^2+1}}\right|^p{\rm d}x=\int_{\mathbb R} \left|\frac{ P_n(x)}{W_{-1}(x)}\right|^p\frac{{\rm d}x}{({x^2+1})^{ p}}\\\lesssim\sup_{x\in \mathbb R} \left|\frac{ P_n(x)}{W_{-1}(x)}\right|^p=  E^p_n(\infty, W_{-2}).
\end{multline*}

Since 
$W_{-2}\gtrsim W^{1/2}$, we have 
$$\log E_n(p,W)\lesssim \log E_n(\infty,  W^{1/2}).$$
Since the function $\vphi/2$ grows regularly, we can use Theorem \ref{main} for the case $p=\infty$ (which is already proved) to obtain that
$$\log E_n(p,W)\lesssim -\int_0^1 \min\left( \frac 12 \vphi\left(\frac 1x\right),n\right){\rm d}x\lesssim-\int_0^1 \frac 12\min\left(  \vphi\left(\frac 1x\right),n\right){\rm d}x,$$
which gives the upper estimate in Theorem \ref{main} for $1\le p <\infty$.

\subsubsection{The lower bound}

Now, let $P_n$ be the extremal   polynomial of degree $n$ for the $L^p$ norm,   such that $P_n(i)=1$ and $$E_n(p,W)=\|P_n\|_{p,W_1}.$$
Then 
\begin{equation}\label{4}
    0\le\int_{\mathbb R} \frac{\log|P_n(x)|{\rm d} x}{x^2+1}=\left(\int_{|x|<A_n}+\int_{|x|\ge A_n}\right)\frac{\log|P_n(x)|{\rm d} x}{x^2+1}=:I_1+I_2, 
\end{equation}
where $A_n=\vphi^{-1}(n)$ as before.

To estimate  the first integral we use the Jensen inequality:
\begin{multline}\displaystyle\label{5}
I_1= 
\int_{-A_n}^{A_n} \log{\displaystyle\left|\frac{P_n(x)}{W(x)\sqrt{x^2+1}}\right|^p}\frac{{\rm d} x}{x^2+1}
+2p\int_{0}^{A_n} \left(\vphi(x)+\log\sqrt{x^2+1}\right)\frac{{\rm d} x}{x^2+1}\\
\lesssim  \log \left(\int_{-A_n}^{A_n} \left|\frac{P_n(x)}{W(x)\sqrt{x^2+1}}\right|^p\frac{{\rm d} x}{x^2+1}\right)+ 
\int_{0}^{A_n} \frac{\vphi(x){\rm d} x}{x^2+1}\\
\lesssim \log{E_n(p,W)}+ 
\int_{0}^{A_n} \frac{\vphi(x){\rm d} x}{x^2+1}.
\end{multline}

To estimate $I_2$ we  use the following lemma. 

\begin{lemma}\label{X}
Let $Q$ be a polynomial  of degree $n$, $a>1$ and $p \ge 1.$ Then
$$\max_{[-a,a]} |Q|^p\lesssim \frac {n^2}a{\int_{-a}^a|Q|^p{\rm d}x}$$
\end{lemma}

\begin{proof}
Let $x_0\in J:=[-a,a]$ be such that $$\max_{J} |Q|=|Q(x_0)|.$$
Then for every  $ \displaystyle x\in J_1:=\left[x_0-\frac{a}{2n^2}, x_0+\frac{a}{2n^2}\right]\cap J$ there exists $\xi$ on the interval $J\cap J_1$ such that
$$|Q(x)-Q(x_0)|=|x-x_0|\cdot|Q'(\xi)|\le  \frac{a}{2n^2}|Q'(\xi)|
$$
Therefore, applying the classical Markov inequality (\cite{Markov}; \cite[Theorem 5.1.8]{Borwein})  on the interval $[-a,a]$, we obtain
$$|Q(x)-Q(x_0)|\le \frac{a}{2n^2}|Q'(\xi)|\le \frac{a}{2n^2}  \frac{n^2}{a}|Q(x_0)|
\le \displaystyle  \frac 12 |Q(x_0)|,$$
and, hence, we have
$$|Q(x)|\ge \displaystyle\frac 12|Q(x_0)|,\;\;|x-x_0|\le \frac a{2n^2}.$$

Thus,
$$\int_{-a}^a|Q|^p{\rm d}x\ge \int_{J\cap J_1}|Q|^p{\rm d}x\ge \frac{a}{4n^2} \frac 1{2^p}|Q(x_0)|^p,$$
proving the lemma.
\end{proof}

Finally, we turn to estimating $I_2$. By the Tchebyshev inequality,
\begin{align*}\displaystyle\label{2}
I_2\lesssim &\int_{A_n}^\infty \frac{\log{|T_n(x/A_n)|}{\rm d} x}{x^2+1}+\int_{A_n}^\infty  \log \left(\max_{[- A_n,A_n]}|P_n|\right)\frac {{\rm d}x}{x^2+1}
\\\lesssim &  \int_{A_n}^\infty \frac{n\cdot\log(x/A_n){\rm d} x}{x^2+1}
+\log \left(\max_{[-A_n,A_n]}|P_n|\right) \int_{A_n}^\infty  \frac {{\rm d}x}{x^2+1}.\end{align*}
By Lemma \ref{X}, we obtain that
\begin{align*}
I_2\lesssim &\frac n{A_n}+ 
\frac 1{A_n}\cdot\log\left(\frac {n^2}{A_n}{\int_{-A_n}^{A_n}|P_n|^p{\rm d}x}\right)\\
\lesssim&
\frac n{A_n}+ \frac {\log n^2}{A_n}
+\frac {1}{A_n}\log\left(\frac{W^p(A_n)({A_n^2+1})^{p/2}}{A_n}\int_{-A_n}^{A_n}\left|\frac{P_n(x)}{W(x)\sqrt{x^2+1}}\right|^p{{\rm d} x}\right)
\\\lesssim&
\frac n{A_n}+ 
\frac {p\log(A_n^2+1)-2\log A_n}{A_n}+
\frac 1{A_n}
\log\left(W^p(A_n)\int_{-A_n}^{A_n}\left|\frac{P_n(x)}{W(x)\sqrt{x^2+1}}\right|^p{{\rm d} x}\right)
\\\lesssim&
\frac n{A_n}+ \frac {\vphi(A_n)}{A_n}+
\frac 1 {A_n}\log E_n(p,W)\\
\lesssim &\frac n{A_n}+\log E_n(p,W),
\end{align*}
because $n=\vphi(A_n).$

Combining this estimate with (\ref{4}) and (\ref{5}) we get
\begin{align*}
- \log{E_n(p,W)} \lesssim\int_{0}^{A_n} \frac{\vphi(x){\rm d} x}{x^2+1}
+\frac{\vphi(A_n)}{A_n}
\end{align*}
 which completes the proof of  Theorem \ref{main2} and, hence, of Theorem \ref{main}, in the case $1~\le~p~<~\infty.$
\qed

\bigskip
\medskip

\noindent Department of Mathematics and Mechanics,
St. Petersburg State University, St. Petersburg, Russia
\newline {\tt a.kononova@spbu.ru}


\begin{thebibliography}{A}

\bibitem{AD} E.~Abakumov, E.~Doubtsov,  Approximation by proper holomorphic maps and tropical power series, Constr. Approx. 47 (2018), no. 2, 321--338.


\bibitem{A} N.~I.~Akhiezer, On the weighted approximation of continuous functions by polynomials
on the entire real axis; Uspekhi Mat. Nauk (N.S.) 11 (4) (1956), (70), 33--43; Amer. Math.
Sot. Trans. Ser. 2, 22 (1962), 95--137.

\bibitem{Bern} S.~N.~Bernstein, Le probl\`eme de l'approximation des fonctions continues sur tout l'axe r\'eel et
l'une de ses applications, Bull. Math. Soc. France, 52 (1924), 399--410.

\bibitem{BKS} A.~Borichev, A.~Kononova, M.~Sodin, Notes on the Szeg\H{o} minimum problem. I. Measures with deep zeroes, arXiv:1902.00874, to appear in Israel Journal  of Mathematics.

{
\bibitem{Borwein} P. Borwein, T. Erdélyi, 
Polynomials and Polynomial Inequalities, Springer-Verlag, New York, 1995.
}


\bibitem{Dzr2} M.~M.~Dzhrbashyan, Biorthogonal systems of rational functions and best approximation of the Cauchy kernel on the real axis,
Mathematics of the USSR-Sbornik, 24:3 (1974) 409--433.

\bibitem{Hall}  T.~Hall, Sur l'approximation polyn\^omiale des fonctions continues d'une variable r\'eelle, Neuvi\`eme Congr\`es des Math\'ematiciens  (1938), Helsingfors, 1939, 367--369.

\bibitem{Koosis}  P.~Koosis, The Logarithmic Integral I, Cambridge University Press, Cambridge, 1988.

\bibitem{LL} A.~L.~Levin, D.~S.~Lubinsky, Christoffel functions, orthogonal polynomials, and Nevai's conjecture for Freud weights, Constr. Approx., 8 (1992), 463--535.

\bibitem{L} D.~Lubinsky, Survey of Weighted Polynomial Approximation with Exponential Weights,  Surv. Approx. Theory 3 (2007), 1--105. 

\bibitem{Markov}
A.A.~Markov, On a problem of D.I.~Mendeleev. Zap. Imp. Akad. Nauk. 62 (1889), 1--24 (In Russian).

\bibitem{M} S.N.~Mergelyan, Weighted approximation by polynomials, Uspehi Mat. Nauk, 11 (1956), 107--152;  Amer. Math. Soc. Translations, 10 (1958), 59--106.

\bibitem{Merg} S.N.Mergelyan, Best approximations with the weight on a straight line, Dokl. Akad. Nauk SSSR, 132:2 (1960), 287--290; Sov.
Math. 1 (1960), 552--556. 


\bibitem{Mha} H.~N.~Mhaskar, Weighted Polynomial Approximation, 
Journal of Approximation Theory, 46 (1986), 100--110.

\bibitem{V} V.S.Videnskii, On normally growing functions, Meetings of the Moscow Mathematical Society, Uspekhi Mat. Nauk,  9 (1954), 2 (60), 207--213 (In Russian). 


\bibitem{Wahde}G.~Wahde, An extremal problem related to Bernstein's approximation problem, Math. Scand., 15 (1964), 131--141. 

\end{thebibliography}
\end{document}